\documentclass[11pt]{article}
\usepackage{a4}
\usepackage{amsfonts}
\usepackage{amssymb}
\usepackage{amsmath}
\usepackage{amsthm}
\usepackage{latexsym}
\usepackage{layout}
\usepackage[english]{babel}
\usepackage{epsfig}
\numberwithin{equation}{section}

\newtheorem{assumptions}{Assumption}[section]
\newtheorem{theorem}[assumptions]{Theorem}

\newtheorem{definition}[assumptions]{Definition}
\newtheorem{proposition}[assumptions]{Proposition}
\newtheorem{corollary}[assumptions]{Corollary}
\newtheorem{lemma}[assumptions]{Lemma}
\newtheorem{remark}[assumptions]{Remark}

\def\1{\raisebox{2pt}{\rm{$\chi$}}}

\newcommand{\eps}       {\epsilon}

\newcommand{\nada}[1]   {}

\newcommand{\R}         {\ensuremath{\mathbb R}}

\def\R{I\!\!R}

\def\1{\raisebox{2pt}{\rm{$\chi$}}}
\def\div{{\rm div}}

\begin{document}

\title{Parabolic equations in time dependent domains}

\author{Juan Calvo\footnote{Departamento de Matem\'atica Aplicada, Universidad de Granada, Granada, Spain,
        e-mail: juancalvo@ugr.es}  
\and Matteo Novaga\footnote{
Department of Mathematics, University of Pisa,
Pisa, Italy, e-mail: matteo.novaga@unipi.it} 
\and Giandomenico Orlandi\footnote{
Department of Computer Science, University of Verona,
Verona, Italy, e-mail: giandomenico.orlandi@univr.it}}

\date{\it Dedicated to the memory of Vicent Caselles}

\maketitle

\begin{abstract}
We show existence and uniqueness results for nonlinear parabolic equations
in noncylindrical domains with possible jumps in the time variable.
\end{abstract}

\noindent {\bf Keywords:} Nonlinear parabolic equations, Noncylindrical domains, Monotone operators.

\noindent {\bf AMS subject classification:} 35K55, 35K20, 35K65, 47H05.

\tableofcontents

\section{Introduction}\label{secintro}

In recent years there has been a renewed interest in problems related with 
partial differential equations formulated in domains that change over time. 
This is partly due to the fact that a number of problems in mathematical 
biology are naturally posed on growing domains (e.g. developing organisms 
or proliferating cells, see for instance \cite{Chaplain, Kondo,  Harrison}) or domains that evolve in some 
particular way. Such issues have originated a wide amount of mathematical research, let us mention \cite{Padilla,Barreira,Maini, Maini2}. To this we should add more classical engineering 
applications like fluids or gases in settings as 
channels or pipes with confining walls that may be displaced, 
removed or brought in at will. A sample of different applications of partial differential equations in evolving domains can be found in the recent survey paper \cite{Knobloch}. In fact, it is not hard 
to conjecture that new applications will involve 
equations in moving domains in the future. Apart from that, 
partial differential equations posed on non-cylindrical domains 
are interesting also from the purely mathematical point of view. 

This has led to an outburst of works on this subject in the 
literature that added up to some classical works  
\cite{Baiocchi,Lieberman86, Lieberman97,Yamada,Lions57, Cannarsa} 
to the extent that the number of 
current references is overwhelming.  
 Let us comment on this literature according to the approach, 
the assumptions on the evolution of the domain where the equation 
is posed and the types of equations considered. Many authors used semigroup methods to tackle
these problems (see for instance \cite{Acquistapace, Lumer}
and references therein), 
but other approaches include adding a time viscosity \cite{Bonaccorsi}, 
mapping the spacetime domain to a cylindrical domain \cite{Attouch} 
or use De Giorgi's minimizing movements \cite{Gianazza, Bernardi}. 
As regards time variations of the domain, it is customary 
to impose some sort of continuity (for instance Lipschitz 
continuity \cite{Savare97}, relaxed to H\"older continuity 
in \cite{Bonaccorsi} and to absolute continuity in \cite{Paronetto}), 
alternatively a monotonicity condition can be used 
(i.e. expanding domains \cite{Gianazza, Bernardi}) 
or Reinfenberg-type domains can be considered \cite{Byun}. 
Concerning the type of equations, most of the works focus 
on parabolic equations which are assumed to be linear 
or in divergence form (see however 
\cite{Bertsch, Lumer, Bonaccorsi, Paronetto} 
where also other operators are admitted).

\smallskip

In this paper we are interested in well-posedness of 
parabolic equations in divergence form, in bounded domains 
that evolve in time. More precisely, we deal with the 
Cauchy--Dirichlet problem, in a formulation that allows 
boundary conditions to depend on time. 

Let us discuss what 
are the novelties of this work with respect to the 
already existing literature. First, we introduce a
simple approach to construct solutions, which consists 
in performing a time slicing of the domain, and then solve 
a family of approximating equations in cylindrical domains. 
The simplicity of this approach may allow to use it as a
starting point for devising numerical methods for this sort of problems. 
Despite its simplicity, we are not aware of other works where 
such a slicing strategy is used. Our approach allows to deal with nonlinear equations, which include the parabolic $p$-Laplacian as a particular case. 
Also, our slicing technique applies to quite  
general variations on the domain over time: we only require 
them to be of bounded variation, allowing for sudden 
{\it jumps} (expansions or contractions) of the domain. 
In particular, we do not impose any constraint 
on the topology of the evolving domains, which may differ from that 
of the initial domain. We are also able to prove uniqueness 
under some additional constraints on the domain (see Section \ref{so_unique}). 

To our best knowledge, this generality 
has not been previously achieved in the literature, except for the 
case of purely expanding domains \cite{Bernardi}. However, in \cite{Paronetto}
F. Paronetto proposes a different approach, 
which can be extended to cover quite general operators 
and boundary conditions.

\paragraph{Possible extensions.}
Since our main goal is presenting a method to tackle parabolic equations in moving domains, we did not focus on looking for the most general possible result. For instance, 
for the sake of simplicity
we chose to deal only with bounded initial data. We stress that our main idea is to use a time slicing to approximate the original problem by a sequence of problems defined on cylindrical domains. 
As we do not focus on any particular equation, we chose to use abstract Lions' theory to provide existence for the approximating problems. 
However, we could also use other theories as starting point 
to provide existence of approximate solutions. 
If we are interested in a particular equation (the p-Laplace equation, say) 
then we will likely be using specific existence results to set up our method, 
and those will provide a much more accurate framework for the admisible set of initial conditions.

In that line of thought, the fact that our present formulation does not allow to deal with degenerate equations, such as the porous media equation and its variants, could appear as a drawback. Again, we argue that suitable modifications of the method here proposed would allow to tackle these problems. In fact, even sticking to Lions' theory, porous media equation and related ones can be treated by making use of the compactness results by Dubinskii \cite{Dubinskii}, carefully adapting our arguments in order to cope with that (see \cite[Chapter I, 12]{LionsBook}). 
We did not pursue this line here in order to keep the presentation as simple as possible.

We also point out that we cannot deal with operators with linear growth 
such as the total variation flow or the parabolic minimal surface equation
(see \cite{Bertsch} for some results in this direction 
in the one-dimensional case). This is another challenging line to explore. Finally, following the same approach it should be possible 
to consider similar evolution equations on manifolds evolving 
in time and/or nonlocal operators (see \cite{Alphonse,Alphonse2} and references therein).

\smallskip

\paragraph{Acknowledgements.} 
We wish to thank the anonimous referee for useful comments and suggestions which helped us to improve the paper.

J.C. was supported in part by ``Plan Propio de Investigaci\'on, programa 9'' (funded by Universidad de Granada and FEDER funds) and Project MTM2014-53406-R (funded by the Spanish MEC and european FEDER funds); he also acknowledges partial support during the development of this work from Obra Social La Caixa through the ``Collaborative Mathematical Research'' programme and a Juan de la Cierva fellowship of the Spanish MEC. 
M.N. was supported by the Italian GNAMPA, and by the University of Pisa via grant PRA-2015-0017. 
M.N. and G.O. were also supported by the Project Italy-Spain MIUR-MICINN IT0923M797.

\smallskip

This problem was proposed to us by our friend and collegue Vicent Caselles. Without his contribution and suggestions
this work would not have been possible, and we dedicate it to his memory.

\section{Standing assumptions and main results}

Our purpose is to prove existence and uniqueness results for nonlinear parabolic equations with time dependent coefficients
in time dependent domains. More precisely, 
given an open set $\widetilde\Omega\subset [0,T]\times\R^d$
we shall consider the following problem:
\begin{equation}\label{model1}
\left\{
\begin{array}{lr}
\displaystyle
u_t(t,x) = \mathrm{div} \left( A(t,x,u,\nabla u)\right)  \qquad & \hbox{in $\widetilde{\Omega}$,}
\\ \\
\displaystyle u(0,x) = u_0(x) \qquad &\hbox{in $\Omega(0)$,}
\\ \\
\displaystyle u(t,x) = \psi(t,x) \qquad & 
\hbox{in $\partial_{l} \widetilde{\Omega}\cup \partial_{-1} \widetilde{\Omega}$}, 
\end{array}
\right.
\end{equation}
where we let $\nu^{\widetilde{\Omega}}=(\nu_t,\nu_x)$ be the outer unit normal to $\partial\widetilde{\Omega}$, $\Omega(0)$ is the initial domain  
defined in Assumption \ref{ass2}, and we set
\begin{align*}
\partial_{\pm 1} \widetilde{\Omega} &:= \{(t,x)\in \partial\widetilde{\Omega}: 
t>0, \nu_t = \pm 1\},
\\
\partial_{l} \widetilde{\Omega} &:= \{(t,x)\in \partial\widetilde{\Omega}: 
|\nu_t | < 1\}
=\{(t,x)\in \partial\widetilde{\Omega}: |\nu_x| > 0\}.
\end{align*}

In order to establish existence and uniqueness of solutions,
we shall make suitable assumptions on the flux vector field  $A$, 
on the data $u_0$, $\psi$ and on the domain $\widetilde{\Omega}$.

\begin{assumptions}
\label{ass1}
The set $\widetilde{\Omega}\subset (0,T)\times \R^d$ is a bounded open set with Lipschitz boundary, and 
we let
$$
\Omega(t) := \{x\in\R^d:\,(t,x)\in \widetilde\Omega\}\qquad t\in (0,T).
$$
Note that $\Omega(t)$ is an open set, possibly empty, for all $t\in (0,T)$.  
\end{assumptions}

Notice that $\Omega(t)$ has Lipschitz boundary 
for a.e. $t\in (0,T)$, and there exist the limits
\begin{equation}
\label{cerodef}
\Omega(t\pm) := \lim_{s\to t\pm}\Omega(s) \qquad \text{for all $t\in [0,T]$,} 
\end{equation}
where the limit is taken in the Hausdorff topology.

\begin{assumptions}
\label{ass2}
The set $\Omega(0):=\Omega(0+)$ is open and has Lipschitz boundary. 
\end{assumptions}

Next we describe our assumptions on the operator $A$.
Let $Q_0$ be an open set of $\R^d$ such that $\cup_{t\in [0,T]} \Omega(t)\subset\subset Q_0$ --where by $\subset\subset$ we mean that the inclusion is compact-- and let $Q_T := (0,T) \times Q_0$. We shall denote by $\mathcal{M}(Q_T)$ the space of all Radon measures on $Q_T$.

\begin{assumptions}
\label{ass3}
The function $A:Q_T \times\R \times\R^d\to \R^d$ is a Carath\'eodory map satisfying
\begin{equation}\label{L1}
\vert A(t,x,z,\xi)\vert \leq c \vert\xi\vert^{p-1} + b(t,x), \quad \hbox{\rm $c > 0$,
$b\in L^{p'}(Q_T)$, $1 < p < \infty$, $\frac{1}{p}+\frac{1}{p'}=1$,}
\end{equation}
\begin{equation}\label{L2}
 A(t,x,z,\xi)\cdot \xi \geq \alpha \vert\xi\vert^{p} - d(t,x), \quad \hbox{\rm $\alpha > 0$, $d\in L^{1}(Q_T)$,}
\end{equation}
\begin{equation}\label{L3}
 (A(t,x,z,\xi)-A(t,x,z,\xi^*))\cdot (\xi-\xi^*) \geq 0, 
\end{equation}
for a.e. $(t,x)\in Q_T$ , and for all  $z\in\R$, $\xi,\xi^*\in\R^d$.
Moreover, we assume that 
\begin{equation}\label{L4}
\vert A(t,x,z,\xi)-A(s,y,w,\xi)\vert \leq 
(\omega(|t-s|+\vert x-y\vert) + C|z-w|) |\xi|^{p-1} ,
\end{equation}
where $\omega$ is a modulus of continuity and $C\ge 0$.
We assume also that
\begin{equation}\label{L5}
A(t,x,z,0)=0\quad \forall z \in \R,\ a.e. \ (t,x)\in Q_T.
\end{equation}
\end{assumptions}
\noindent Note that \eqref{L3} and \eqref{L5} imply that
\begin{equation}
\label{strong_monotonicity}
 A(t,x,z,\xi)\cdot \xi  \geq 0, \quad \hbox{\rm a.e. in $Q_T$, and for all
 $z\in\R$, $\xi \in \R^d$.}
\end{equation}

We will consider the problem \eqref{model1} with initial and boundary conditions
\begin{equation}\label{BD1}
u(0,x) = u_0(x)  \in L^\infty(\Omega(0)),
\end{equation}
\begin{equation}\label{BD2}
u(t,x) =\psi(t,x), \qquad \hbox{\rm $(t,x)\in \partial_{l} \widetilde{\Omega}\cup \partial_{-1} \widetilde{\Omega}$, $t> 0$.}
\end{equation}

\begin{assumptions}
\label{ass4}
We assume that 
\begin{equation}\label{psi1case1}
\psi \in C({Q_T}) \cap L^p(0,T;W_0^{1,p}(Q_0)),
\end{equation}
and
\begin{equation}\label{psi2case1}
\psi_t \in L^1({Q_T}) \cap L^{p'}(0,T;W^{-1,p'}(Q_0)). 
\end{equation}
\end{assumptions}
Let us now define the space after which we model the solutions of our problem.
 
\begin{definition}\label{here}
Let $\mathcal{V}$ be the closure of $C_c^1(\widetilde\Omega)$ with respect to the norm
$$
||v||_{\mathcal{V}}:= \left(\int_{\widetilde\Omega} |\nabla v|^p\, dxdt\right)^{1/p}, \quad v \in C_c^1(\widetilde\Omega).
$$ 
\noindent Notice that functions in $\mathcal{V}$ do not necessarily
have zero trace on $\partial_{\pm 1}\widetilde\Omega$ or on $\Omega(0)$.
\end{definition}

Our concept of solution will be the following:
\begin{definition}
\label{solucion_distribucional} 
We say that a function $u\in L^1(\widetilde\Omega)$ is a weak solution of (\ref{model1}) if the following statements hold:
\begin{enumerate}
\item 
$u-\psi \in \mathcal{V}$ and $A(t,x,u,\nabla u)\in L^{p'}(\widetilde\Omega)$. 

\item $u_t \in \mathcal{V}^*$ 
(note that this implies that $u$ has a trace on $\partial_{\pm 1} \widetilde\Omega$ and on $\Omega(0)$).

\item {$u(0) =u_0$ a.e. on $\Omega(0)$ and $u=\psi$ a.e. on every relatively open subset of $\partial_{-1} \widetilde\Omega$.}

\label{solut3}
\item The following integral formulation
\begin{equation}\label{weakId1}
- \int_0^T \int_{\Omega(t)} u\phi_t \, dxdt -  \int_{\Omega(0)} u_0 \phi(0)\, dx + \int_0^T \int_{\Omega(t)} A(t,x,u,\nabla u)\cdot\nabla \phi \, dxdt = 0
\end{equation}
holds for all $\phi\in\mathcal{D}([0,T)\times Q_0)$ with $\mathrm{supp}\, \phi\subset\subset \widetilde\Omega$.
\end{enumerate}

\end{definition}

Let us state the main existence result of this paper.
\begin{theorem}
\label{main_result}
Let Assumptions \ref{ass1}--\ref{ass4} be satisfied. 
Then there exists a weak solution of \eqref{model1} in the sense of Definition \ref{solucion_distribucional}.
\end{theorem}
Following \cite[Assumption H.2]{Paronetto},
we introduce an additional assumption on the domain which we will need in the 
uniqueness proof.

\begin{assumptions}\label{for_uniqueness}
For every $t_0\in [0,T]$, there exist an open neighborhood $U$ of $t_0$ 
and a family of maps $G(\cdot,t):\Omega(t_0)\to \Omega(t)$, with $t\in U\cap [0,T]$,
such that
\begin{itemize}
\item[--] $G(\cdot,t)$ is a bijection for almost every $t\in U\cap [0,T]$;
\item[--] $G(\cdot,t)$ is Lipschitz continuous with its inverse for every $t\in U\cap [0,T]$;
\item[--] $G(x,\cdot)$ and $|\nabla G(x,\cdot)|$ are absolutely continuous for almost every
$x\in\Omega(t_0)$;
\item[--] $|\nabla G(\cdot,t)|\in L^1(\Omega(t_0))$ for every $t\in U\cap [0,T]$
and $\partial_t |\nabla G|\in L^1(\Omega(t_0)\times (0,T))$.
\end{itemize}
\end{assumptions}
Note that this assumption does not allow "jumps" of the sections $\Omega(t)$.
However, we could work in a more general framework in which the conditions in Assumption \ref{for_uniqueness} break down for a finite set of times; we comment on this in Remark \ref{finite_sing} below.

Let us state now our uniqueness result.
\begin{theorem}\label{teounique}
Let Assumptions \ref{ass1}--\ref{ass4} and \ref{for_uniqueness} be satisfied. Then the solution of \eqref{model1} is unique in the class of weak solutions.
\end{theorem}



\section{Construction of approximate solutions}
\label{Sec:approximations}
Let us divide the interval $[0,T]$ into sub-intervals $0=t_0 < t_1< \ldots < t_{N-1} < t_N=T$.
The points $t_i$ are chosen so that:
\begin{enumerate}

\item $\Omega(t_i)$ has Lipschitz boundary for all $i\in \{0,\ldots,N-1\}$,

\item  (\ref{L1})--(\ref{L5}) hold for a.e. $x \in \Omega(t_i)$ 
and for all $z\in\R$, $\xi\in\R^d$,

\item $t_i$ are Lebesgue points of $\psi(t)\in L^p(0,T;W_0^{1,p}(Q_0))$ and $\psi(t_i)\in W^{1,p}(Q_0)$,

\item $t_i$ are Lebesgue points of the map $t\in [0,T]\to A(t)\in L^1(Q_0\times (-R,R) \times B(0,R))$ for any $R > 0$, being $B(0,R)$ the open ball centered at zero with radius $R$,\label{punto4}

\item $\Delta:=\max_{k=0,\ldots,N-1} |t_k-t_{k+1}|\to 0$ as $N\to \infty$.
\end{enumerate}
Let $I_k = [t_k,t_{k+1})$. We iteratively solve the parabolic problem
\begin{equation}\label{model2}
\left\{
\begin{array}{cl}
\displaystyle
u^k_t = \mathrm{div} \left( A(t_k,x,u^k,\nabla u^k)\right),  
& t \in I_k,\ x\in \Omega(t_k)
\\ \\
\displaystyle u^k(t,x) 
= \psi(t,x), & t \in I_k,\ x\in \partial \Omega(t_k)
\\ \\
\quad \quad u^k(t_k,x) = & \left\{ 
\begin{array}{lr}
\lim_{t\to t_k-}u^{k-1}(t,x), &  x\in\Omega(t_{k})\cap \Omega(t_{k-1})
\\ \\
\psi(t_k,x), & x\in\Omega(t_{k})\backslash \Omega(t_{k-1}).
\end{array}
\right.
\end{array}
\right.
\end{equation}
If $t_0=0$ we let $u^0(0,x) = u_0(x)$.
Notice that the iterative initial condition for $t=t_k$ makes sense thanks to the continuity properties of $u^{k-1}$, see \eqref{cont}.


\subsection{Study of the model problem on a time slice}

Let $\Omega_0$ be an open bounded set in $\R^d$ with Lipschitz boundary.
Let $A(x,z,\xi)$ be such that (\ref{L1})--\eqref{L5} hold a.e. in $x \in \Omega_0$ and for all $z\in\R$, $\xi\in\R^d$.
Let us consider the problem
\begin{equation}\label{model2S}
\left\{
\begin{array}{lc}
\displaystyle
u_t = \mathrm{div} \left( A(x,u,\nabla u)\right)  \qquad \hbox{$t \in [0,T]$, $x\in \Omega_0$,}\\ \\
\displaystyle u(t,x) = \psi(t,x) 
\qquad \qquad \hbox{$t \in [0,T]$, $x\in \partial \Omega_0$,}
\\ \\
\displaystyle u(0,x) = u_0(x)
\qquad \qquad \hbox{$x\in \Omega_0$.}
\end{array}
\right.
\end{equation}
where $\psi$ satisfies
\eqref{psi1case1}--\eqref{psi2case1} and $u_0\in L^2(\Omega_0)$.

\begin{definition}\label{d2.1}
We say that a function $u\in L^1((0,T)\times \Omega_0)$ 
is a weak solution of (\ref{model2S}) if
$u\in L^p(0,T;W^{1,p}(\Omega_0))$, $A(x,u,\nabla u)\in L^{p'}((0,T)\times \Omega_0))$,
\begin{equation}\label{weakId2}
- \int_0^T \int_{\Omega_0} u\phi_t\, dxdt - \int_0^T \int_{\Omega_0} u_0 \phi(0)\, dx + \int_0^T \int_{\Omega_0} A(x,u,\nabla u)\cdot\nabla \phi \, dxdt= 0
\end{equation}
holds for all $\phi\in\mathcal{D}([0,T)\times \Omega_0)$, and
$$
u(t)-\psi(t) \in W_0^{1,p}(\Omega_0) \qquad \hbox{a.e. $t\in (0,T)$.}
$$
\end{definition}

Note that, by (\ref{L1}), if $u\in L^p(0,T;W^{1,p}(\Omega_0))$, then $A(x,u,\nabla u)\in L^{p'}((0,T)\times \Omega_0))$.

\begin{proposition}
\label{Lionspge2}
Problem \eqref{model2S} admits a unique weak solution in the sense of Definition \ref{d2.1}. 
\end{proposition}
\begin{proof}
The proof is a standard application of the theory developed in 
\cite{LionsBook,LionsPaper}; we include it for completeness. 
We consider the auxiliary problem
\begin{equation}\label{model2Sv}
\left\{
\begin{array}{ll}
\displaystyle
v_t - \mathrm{div} \left( \widetilde{A}(t,x,v,\nabla v)\right)=- \psi_t  &\qquad \hbox{$t \in [0,T]$, $x\in \Omega_0$,}
\\ \\
\displaystyle v(t,x) = 0 
\qquad \qquad & \hbox{$t \in [0,T]$, $x\in \partial \Omega_0$,}
\\ \\
\displaystyle v(0,x) = u_{0}(x)-\psi(0,x)
\qquad \qquad & \hbox{$x\in \Omega_0$.}
\end{array}
\right.
\end{equation}
Here 
$$
\widetilde{A}(t,x,z,\xi) := A(x,z+\psi(t,x),\xi+\nabla \psi(t,x)).
$$
According to the notation in \cite{LionsBook,LionsPaper}, we let $H=L^2(\Omega_0)$, 
$$B=\left\{
\begin{array}{ll}
W_0^{1,p}(\Omega_0) & \mbox{if}\quad p\ge 2, 
\\ \\
W_0^{1,p}(\Omega_0) \cap L^2(\Omega_0) & \mbox{if}\quad 1<p<2,
\end{array}
\right.
$$ 
and $F=L^p(0,T;B)$, so that $B$ is dense in $H$ and 
$$
-\div \widetilde{A}: F \rightarrow F'=L^{p'}(0,T;B') \quad \mbox{and} \quad 
\psi_t \in F',
$$
with
$$B'=\left\{
\begin{array}{ll}
W^{-1,p'}(\Omega_0) & \mbox{if}\quad p\ge 2, 
\\ \\
W^{-1,p'}(\Omega_0) + L^2(\Omega_0) & \mbox{if}\quad 1<p<2.
\end{array}
\right.
$$ 
 Observe that, by our assumptions on $A(x,z,\xi)$ and $\psi$, 
$\widetilde{A}(t,x,z,\xi)$
is a Leray-Lions operator (see \cite{LionsBook,LionsPaper}). Indeed, the monotonicity requirement is satisfied thanks to \eqref{L3}. The coercivity condition follows from \eqref{L2} and Poincare's inequality in a standard way (lower order terms are estimated thanks to \eqref{L1}). Then thanks to Lions' theory there exists some $v \in F$ solving \eqref{model2Sv} in $F'$. In fact, this solution verifies that
$$
v \in L^p(0,T,B)\quad \mbox{and} \quad v_t \in  L^{p'}(0,T;B').
$$
Notice that $v \in C(0,T;L^2(\Omega_0))$ thanks to Lemma \ref{interpol} below.

Let now $u= v+\psi$. Then $u$ is a weak solution of (\ref{model2S}) with initial condition $u(0)=u_{0}$. Clearly $u \in L^p(0,T;W^{1,p}(\Omega_0))$, $u_t \in  L^{p'}(0,T;B')$
and
\begin{equation}\label{cont}
u \in C(0,T;L^2(\Omega_0)).
\end{equation}

To prove uniqueness let $u,v$ be two different solutions. Note that $u-v\in F$.
If $A$ does not depend on $u$, we multiply the equation for $(u-v)_t$ by $u-v$ and integrate by parts (see e.g. \cite[Chapter III]{Showalter}). Recalling \eqref{L3} we have that

\begin{eqnarray*}
\frac{1}{2} \frac{d}{dt} (u-v,u-v)_H &=& \langle (u-v)_t,u-v \rangle_{W^{-1,p'}(\Omega_0)-W_0^{1,p}(\Omega_0)}
\\
&=& \langle  \div (A(x,\nabla u)-A(x,\nabla v)), u-v\rangle_{W^{-1,p'}(\Omega_0)-W_0^{1,p}(\Omega_0)} 
\\
&=& - \langle A(x,\nabla u)-A(x,\nabla v), \nabla u-\nabla v\rangle_{L^{p'}(\Omega_0)-L^{p}(\Omega_0)}
\\
&\le& 0\,.
\end{eqnarray*}
Hence $\|u-v\|_2$ is nonincreasing and uniqueness follows. 

For the general case, consider $\delta>0$ and let
$$
T_\delta(s) = \left\{
\begin{array}{cc}
s & \mbox{if}\ -\delta \le s\le \delta,
\\
-\delta & \mbox{if}\ s<-\delta,
\\
\delta & \mbox{if}\ s>\delta. 
\end{array}
\right.
$$
Clearly $T_\delta(u-v)\in F$ and again after multiplication of the equation for $(u-v)_t$ by $T_\delta(u-v)/\delta$ and integration by parts we obtain that
$$
\frac{1}{2} \frac{d}{dt} (u-v,T_\delta(u-v)/\delta)_H =  - \frac{1}{\delta}\langle \nabla T_\delta(u - v), A(x,u,\nabla u)-A(x,v,\nabla v)\rangle_{L^{p}(\Omega_0)-L^{p'}(\Omega_0)}
$$
$$
\le - \frac{1}{\delta}\langle (\nabla u - \nabla v)\chi_{\{|u-v|\le \delta\}}, A(x,u,\nabla v)-A(x,v,\nabla v)\rangle_{L^{p}(\Omega_0)-L^{p'}(\Omega_0)}.
$$
Then, using \eqref{L4} we get
$$
\frac{1}{2} \frac{d}{dt} (u-v,T_\delta(u-v)/\delta)_H \le \frac{C}{\delta} \int_{\{|u-v|\le \delta\}} |u-v| |\nabla u -\nabla v| |\nabla v|^{p-1}\, dx \le C \int_{\{|u-v|\le \delta\}} |\nabla u -\nabla v| |\nabla v|^{p-1}\, dx.
$$
The term on the far right converges to zero when $\delta \to 0$, since the integrand is in $L^1(\Omega_0)$ and $\nabla (u-v)=0$ a.e. where $u-v=0$. As $
T_\delta/\delta(u-v) \rightarrow \mbox{sign}(u-v)$, we get that $\|u-v\|_1$ is non-increasing, and we conclude the proof.
\end{proof}

The following continuity result is standard (see for instance 
\cite[Ch. 2, Rem. 1.2]{LionsBook} or \cite{Showalter}).

\begin{lemma}\label{interpol}
Let $V$ be a reflexive Banach space with dual $V'$. Let $H$ be a Hilbert space that we identify with its dual. Assume that $V\subset H \subset V'$ with the injection $V\subset H$ being dense. Then, $u \in L^p(0,T;V)$ together with $u_t\in L^{p'}(0,T;V')$ imply that there is a representative of $u$ which is continuous from $[0,T]$ to $H$.  
\end{lemma}

Since $u^k\in C(t_k,t_{k+1};L^2(\Omega(t_k)))$ we can define the traces
$$
u^k(t_k+) := \lim_{t\to t_k+} u^k(t)
\qquad 
u^k(t_{k+1}-) := \lim_{t\to t_{k+1}-} u^k(t),
$$ 
where the limit is taken in $L^2(\Omega(t_k))$.

\subsection{The approximate solutions $u^\Delta$}\label{sect:together}

We now let
$$
{\Omega}^\Delta:=\{(t,x): t\in [t_k,t_{k+1}), \, x\in \Omega(t_k), \,\, k=0,\ldots,N-1\}
$$
$$
=  \cup_{k=1,\ldots, N-1}  [t_k,t_{k+1}) \times \Omega(t_k).
$$
Notice that ${\Omega}^\Delta$ does not depend only on 
$\Delta=\max_{k=0,\ldots,N-1} |t_k-t_{k+1}|$, but depends on the entire sequence
$\{t_k\}_k$.

\begin{lemma}
\label{basic_fact}
$\Omega^\Delta$ converges to $\widetilde{\Omega}$ in the Hausdorff sense. 
As a consequence
$\chi_{\Omega^\Delta}\rightarrow \chi_{\widetilde{\Omega}}$ strongly in $L^1(Q_T)$ (hence in $L^p(Q_T)$ for all $p<\infty$).
\end{lemma}
\begin{proof}
The Hausdorff convergence of $\Omega^\Delta$ to $\widetilde{\Omega}$ 
can be easily verified when $\widetilde\Omega$ is a polyhedron.
The claim follows by approximating
a generic $\widetilde\Omega$ with Lipschitz boundary
with polyhedra, in the topology generated by the Hausdorff distance. 
\end{proof}
%



We now glue the 
solutions $u^k(t,x)$ of (\ref{model2}) together and define
the approximate solutions
\begin{eqnarray}
\label{udelta}
u^\Delta(t,x) &:=& \sum_{k=0}^{N-1} \chi_{[t_k,t_{k+1})}(t) u^k(t,x)\chi_{\Omega(t_k)}(x),
\\
\label{p7def}
\tilde u^\Delta(t,x) &:=& \sum_{k=0}^{N-1} \chi_{[t_k,t_{k+1})}(t) (u^k(t,x)\chi_{\Omega(t_k)}(x) + \psi(t,x) \chi_{Q_0\setminus \Omega(t_k)}(x)),
\end{eqnarray}
for $(t,x)\in Q_T$. 
When we write $u^k(t,x)\chi_{\Omega(t_k)}(x)$ 
in the above formulae we intend the function which coincides with 
$u^k(t,x)$ in $\Omega(t_k)$ and it is equal to zero outside $\Omega(t_k)$.

In the sequel we shall prove 
the compactness of $u^\Delta$ and $\tilde u^\Delta$
as $\Delta\to 0$.

\subsection{Estimates on $u^\Delta$}\label{sect:BE}

We now derive some estimates on the approximate solutions $u^\Delta$
defined in \eqref{udelta}.

\begin{lemma}
\label{linftybound}
Assume that $\|\psi\|_\infty, \|u_0\|_\infty \le C$ for some $C>0$.
 Then $\|u^\Delta\|_{L^\infty(\Omega^\Delta)} \le C$ for any $t>0$.
\end{lemma}
\begin{proof}
It is enough to prove the estimate in $(0,t_1)\times \Omega(0)$. Let $[\cdot]^+$ denote the positive part (resp. $[\cdot]^-$ the negative part) and let $C\ge \|\psi\|_\infty$. Then the pairing of $[u-C]^+$ with $u_t^\Delta$ makes sense; multiplying \eqref{model2S} by $[u-C]^+$ and integrating by parts we get to
$$
\frac{1}{2} \frac{d}{dt} \int_{\Omega(0)} ([u^\Delta(t)-C]^+)^2\ dx = \int_{\Omega(0)} [u^\Delta-C]^+ \div A(0,x,u^\Delta,\nabla u^\Delta) \ dx
$$
$$
= - \int_{\Omega(0)} A(0,x,u^\Delta,\nabla u^\Delta) \nabla ([u^\Delta-C]^+)\ dx.
$$
There are no boundary terms present thanks to our choice of $C$. 
Note that $\nabla ([u^\Delta-C]^+) = \chi_{\{u^\Delta>C\}} \nabla u^\Delta$, so that we can use \eqref{strong_monotonicity} to ensure that the time derivative above  
is nonpositive. Hence,
$$
\int_{\Omega(0)} ([u^\Delta(t)-C]^+)^2 \ dx \le \int_{\Omega(0)} ([u_0-C]^+)^2 \ dx.
$$
Thus, if $u_0 \le C$ then $u^\Delta(t)\le C$ too for any $t \in [0,t_1)$. This works in the same way for the time derivative of the integral of $([u^\Delta+C]^-)^2$, with inequalities reversed. If we now choose $C = \max\{\|u_0\|_\infty,\|\psi\|_\infty \}$, we deduce that $\|u^\Delta(t)\|_\infty \le C$.
\end{proof}

\begin{lemma}
\label{nabla_bound}
There holds 
$$
\sum_{k=0}^{N-1}\int_{t_k}^{t_{k+1}} \int_{\Omega(t_k)} \vert \nabla  u^\Delta(t)\vert^p \, dx dt  \leq C,
$$
for some constant $C>0$ depending only on $\widetilde \Omega$, on $\psi$
and on the structural constants in Assumption \ref{ass3}.
\end{lemma}

\begin{proof}
We fix $k$ and notice that the pairing of $u^\Delta-\psi$ with $u_t^k$ on $(t_k,t_{k+1})\times \Omega(t_k)$ makes sense. After integration by parts we get
$$
\frac{1}{2} \frac{d}{dt} \int_{\Omega(t_k)} (u^\Delta-\psi)^2 \ dx = -\int_{\Omega(t_k)} \nabla (u^\Delta-\psi) A(t_k,x,u^\Delta,\nabla u^\Delta)\ dx - \int_{\Omega(t_k)}(u^\Delta-\psi) \psi_t\ dx.
$$
Notice that the last term is well-defined
thanks to our assumptions on $\psi$ and to Lemma \ref{linftybound}. Integrating the former equality on $[t_k,t_{k+1}]$,
we obtain
$$
 \frac{1}{2}\int_{\Omega(t_k)} (u^\Delta(t_{k+1})-\psi)^2 \ dx = \frac{1}{2}\int_{\Omega(t_k)} (u^\Delta(t_k)-\psi)^2 \ dx - \int_{t_k}^{t_{k+1}} \int_{\Omega(t_k)} (u^\Delta-\psi) \psi_t \ dxdt 
 $$
 $$
  + \int_{t_k}^{t_{k+1}} \int_{\Omega(t_k)} \nabla \psi A(t_k,x,u^\Delta,\nabla u^\Delta)\ dxdt - \int_0^T \int_{\Omega(t_k)} \nabla u^\Delta A(t_k,x,u^\Delta,\nabla u^\Delta)\ dxdt
$$
$$
=: I +II +III + IV.
$$
Let us now control the last three terms. The second one can be easily 
estimated as
$$
II \le 2 \bar C \int_{t_k}^{t_{k+1}} \int_{\Omega(t_k)} |\psi_t| \ dxdt, \quad \bar C:= \max\{\|\psi\|_\infty,\|u_0\|_\infty\}.
$$
Concerning the fourth term, using \eqref{L2} we get
$$
IV \le -  \int_{t_k}^{t_{k+1}} \int_{\Omega(t_k)} \alpha |\nabla u^\Delta|^{p}\ dxdt + \int_{t_k}^{t_{k+1}} \int_{\Omega(t_k)} |d(t,x)| \ dxdt.
$$
In a similar way, using \eqref{L1} we obtain
$$
III\le  \int_{t_k}^{t_{k+1}} \int_{\Omega(t_k)} c |\nabla \psi|  |\nabla u^\Delta|^{p-1}\ dxdt +  \int_{t_k}^{t_{k+1}} \int_{\Omega(t_k)} |\nabla \psi| b(t,x)\ dxdt= A+B.
$$
Let us estimate $A$ and $B$. For that we use Young's inequality with weights:
$$
 a\, b \le \frac{\eps^pa^p}{p} + \frac{b^{p'}}{p' \eps^{p'}}, \quad \epsilon>0, \mbox{being}\ p, p' \ \mbox{given by}\ \eqref{L1}.
$$
Then
$$
B \le \frac{1 }{p} \Vert \nabla \psi\Vert_{L^p([t_k,t_{k+1}]\times \Omega(t_k))}^p + \frac{1 }{p'} \Vert   b\Vert_{L^{p'}([t_k,t_{k+1}]\times \Omega(t_k))}^p
$$
and
 $$
 A\le \frac{c \epsilon^p}{p} \Vert \nabla \psi\Vert_{L^p([t_k,t_{k+1}]\times \Omega(t_k))}^p +  \frac{ c}{p' \epsilon^{p'}}
\Vert \nabla u^{\Delta}\Vert_{L^p([t_k,t_{k+1}]\times \Omega(t_k))}^{p}
 $$
for any $\epsilon>0$. Let us choose $\epsilon$ so that $c/(p' \epsilon^{p'})=\alpha/2$. Collecting all the estimates, we obtain
$$
\frac{1}{2}\int_{\Omega(t_k)} (u^\Delta(t_{k+1})-\psi)^2 \ dx +\frac{\alpha}{2} \int_{t_k}^{t_{k+1}} \int_{\Omega(t_k)}  |\nabla u^\Delta|^{p}\ dxdt
$$
$$
\le \frac{1}{2}\int_{\Omega(t_k)} (u^\Delta(t_k)-\psi)^2 \ dx + 2 
\bar C \int_{t_k}^{t_{k+1}} \int_{\Omega(t_k)} |\psi_t| \ dxdt +\frac{c \epsilon^p}{p} \Vert \nabla \psi\Vert_{L^p([t_k,t_{k+1}]\times \Omega(t_k))}^p
$$
$$
+\frac{1 }{p} \Vert \nabla \psi\Vert_{L^p([t_k,t_{k+1}]\times \Omega(t_k))}^p + \frac{1 }{p'} \Vert   b\Vert_{L^{p'}([t_k,t_{k+1}]\times \Omega(t_k))}^p + \int_{t_k}^{t_{k+1}} \int_{\Omega(t_k)} |d(t,x)| \ dxdt.
$$
By summing up the previous inequalities from $k=0$ to $k=N-1$, we get 
\begin{eqnarray*}
& & \frac{1}{2}  \int_{\Omega(t_{N-1})} (u^\Delta(t_N)-\psi)^{2}\, dx +
\frac{\alpha}{2} \sum_{k=0}^{N-1}\int_{t_k}^{t_{k+1}} \int_{\Omega(t_k)} \vert \nabla  u^\Delta(t)\vert^p \, dx dt \\ & \leq &
\frac{1}{2}  \int_{\Omega(0)} (u^\Delta(0)-\psi)^{2}\, dx + \frac{1 }{p'} \sum_{k=0}^{N-1} \Vert   b\Vert_{L^{p'}([t_k,t_{k+1}]\times \Omega(t_k))}^p  \\
& + & \frac{1 }{p}\left(1+ c\left(\frac{2c}{\alpha p'}\right)^\frac{p}{p'} 
\right)\sum_{k=0}^{N-1}  \Vert \nabla \psi\Vert_{L^p([t_k,t_{k+1}]\times \Omega(t_k))}^p 
\\ & + &
2\bar C  \sum_{k=0}^{N-1} \int_{t_k}^{t_{k+1}} \int_{\Omega(t_k)}  |\psi_t|  \, dx dt + \sum_{k=0}^{N-1} \int_{t_k}^{t_{k+1}} \int_{\Omega(t_k)}  |d(t,x)|  \, dx dt.
\end{eqnarray*}
On the aid of Lemma \ref{linftybound} the thesis follows.
\end{proof}

Recalling the definition of $\tilde u^\Delta$ and the assumptions on $\psi$,
from Lemma \ref{nabla_bound} we obtain the following result:

\begin{corollary}\label{cor1p}
There exists $C>0$ depending only on $\widetilde \Omega$, on $\psi$
and on the structural constants in Assumption \ref{ass3}, such that 
$$
\|\tilde u^\Delta\|_{L^p(0,T;W^{1,p}(Q_0))}\le C\,.
$$
In particular, the sequence $\{\tilde u^\Delta\}$ 
is weakly relatively compact in $L^p(0,T;W^{1,p}(Q_0))$. 
\end{corollary}

\subsection{Time compactness of $\tilde u^\Delta$}\label{sect:compactness}

We now show a stronger compactness property of $u^\Delta$. 
For this aim, we need the following result, proved in 
\cite{simon1986compact}.

\begin{theorem}
\label{thSimon}
Let $X,B,Y$ be three Banach spaces such that $X\subset B\subset Y$. Assume that $X$ is compactly embedded in $B$ and
\begin{equation}\label{compactS1}
\hbox{\rm $F$ is a bounded set in $L^1(0,T;X)$,}
\end{equation}
\begin{equation}\label{compactS2}
\hbox{\rm $\Vert  \tau_h f - f\Vert_{L^1(0,T-h;Y)}\to 0$ as $h\to 0$, uniformly for $f\in F$,}
\end{equation}
where $(\tau_hf)(t)=f(t+h)$ for $h>0$.
Then $F$ is relatively compact in $L^1(0,T;B)$. 
\end{theorem}

Let
$$\psi^\Delta(t,x) : =  \sum_{k=0}^{N-1} \chi_{[t_k,t_{k+1})}(t) \psi(t,x) 
\chi_{Q_0\setminus \Omega(t_k)}(x)\,=\,\psi(t,x)\chi_{Q_T\setminus \Omega^\Delta}(t,x),
$$
so that we have
$\tilde u^\Delta(t,x)  = u^\Delta(t,x)  + \psi^\Delta(t,x)$.
\begin{lemma}
\label{duality_estimate}
Let $0<k\le N$ be fixed. Then $u^k_t (t)\chi_{\Omega(t_k)}\in L^{p'}(t_k,t_{k+1},W^{-1,p'}(\Omega(t_k)))$ and the following estimate holds:
$$
\|�u^k_t (t)\chi_{\Omega(t_k)}\|_{L^{p'}(t_k,t_{k+1},W^{-1,p'}(\Omega(t_k)))} \le c \Vert u^\Delta \Vert_{L^p(t_k,t_{k+1},W^{1,p}(\Omega(t_k)))}^{p-1} 
+ \Vert b\Vert_{L^p(t_k,t_{k+1},L^{p'}(\Omega(t_k)))}.
$$
\end{lemma}
\begin{proof}
We show the estimate by duality. Define $B_k:=L^{p}(t_k,t_{k+1},W_0^{1,p}(\Omega(t_k)))$ and let $\phi \in B_k$. We compute
$$
\langle u^k_t (t)\chi_{\Omega(t_k)},\phi\rangle_{B_k'-B_k} = - \int_{t_k}^{t_{k+1}} \int_{\Omega(t_k)} {A}(t_k,x,u^k,\nabla u^k(t))\cdot\nabla \phi\, dxdt.
$$
Hence using \eqref{L1}
$$
\left|\langle u^k_t \chi_{\Omega(t_k)},\phi\rangle_{B_k'-B_k} \right| \leq
\int_{t_k}^{t_{k+1}} \int_{\Omega(t_k)} ( c\vert \nabla u^\Delta(t)\vert^{p-1} 
+  {b}(t,x))|\nabla \phi|  dxdt
$$
$$
\leq \int_{t_k}^{t_{k+1}} \left( c \Vert u^\Delta(t) \Vert_{W^{1,p}(\Omega(t_k))}^{p-1} 
+ \Vert b(t)\Vert_{L^{p'}(\Omega(t_k))} \right) \Vert \phi\Vert_{W_0^{1,p}(\Omega(t_k))}\, dt.
$$
The result follows.
\end{proof}
\begin{lemma}\label{comp_time}
The sequence $\{\tilde u^\Delta\}$ is relatively compact {in $L_{loc}^1(\widetilde\Omega)$.}
\end{lemma}
\begin{proof}
We consider a cylinder $C:=[t_1,t_2]\times K\subset \subset \widetilde\Omega$. We want to apply Theorem \ref{thSimon} with  
$f=\tilde u^\Delta|_C=u^\Delta|_C$, 
$X=W^{1,p}(K),\, B=L^1(K)$ and $Y=W^{-1,p'}(K)+L^1(K)$.
Here $Y$ is a Banach space equipped with the norm
$$\|y\|_Y := \inf\{ \|y_1\|_{W^{-1,p'}(K)} + \|y_2\|_{L^1(K)}:\, y_1+y_2=y\}.$$
Then $X\subset B \subset Y$ and $X$ is compactly embedded in $B$.

Notice that, since $C\subset\subset\widetilde\Omega$, we have
$$
\tilde u^\Delta_t|_C =u^\Delta_t|_C= \sum_{k=0}^{N-1} 
\chi_{[t_k,t_{k+1})}(t) u^k_t(t,x)\chi_{\Omega(t_k)}(x)|_C
\qquad \text{
for $N$ large enough.} 
$$

Estimate \eqref{compactS1} directly follows from Lemma \ref{linftybound}.
In order to prove \eqref{compactS2},
we notice that (with a slight abuse of notation) 
$$
\tilde u^\Delta(t+h) - \tilde u^\Delta(t) =\int_t^{t+h} \tilde u^\Delta_t(s) \, ds = \int_t^{t+h} \sum_{k=0}^{N-1} \chi_{[t_k,t_{k+1})}(t) u^k_t(s,x)\chi_{\Omega(t_k)}(x) \, ds:=u_1^\Delta(t,h).
 $$ 

We claim  that
\begin{equation}\label{compactPS1}
\int_{t_1}^{t_2-h} \Vert u_1^\Delta(t,h) \Vert_{W^{-1,p'}(K)} \, dt \to 0 
\qquad \hbox{\rm as $h\to 0+$}
\end{equation}
uniformly in $N$; this would imply \eqref{compactS2}. To prove it we sum up all the estimates coming from Lemma \ref{duality_estimate} for different values of $k$ 
in order to cover the cylinder $C$. We obtain that 
there exist $\tilde{C}>0$ independent of $N$ and $t \in [t_1,t_2]$ 
such that $\|u_1(t,h)^\Delta\|_Y \le \tilde{C} h$, which implies \eqref{compactPS1}. Hence $\tilde u^\Delta$ is strongly compact in $L^1(C)$. Now any compact set in $\widetilde\Omega$ can be covered by a finite number of open cylinders. To conclude we take a countable sequence of compact sets embedded in $\widetilde\Omega$ whose increasing union exhausts $\widetilde\Omega$ and apply a diagonal procedure.
\end{proof}
\begin{corollary}
\label{cor}
There exists a subsequence of $\{\tilde u^\Delta\}$
which converges strongly in $L^1(Q_T)$.
\end{corollary}
\begin{proof}
We can combine Lemma \ref{comp_time} with the uniform bound provided by Lemma \ref{linftybound} to use Lebesgue's dominated convergence theorem. Note that
the functions are constantly equal to $\psi$ outside $\Omega^\Delta$ and that Lemma \ref{basic_fact} applies.
\end{proof}
%

\section{Existence of solutions}\label{sect:limit}

In this section we prove the existence of weak solutions of \eqref{model1}.

\subsection{Convergence of the approximate solutions}

\begin{lemma}\label{u_conv}
There are functions $\tilde u, u$ such that the following statements hold 
(up to extracting a subsequence) for $N \to \infty$:
\begin{enumerate}
\item[1)] $\tilde u^\Delta \rightharpoonup \tilde u$ weakly in $L^p(0,T,W^{1,p}(Q_0)) \cap L^\infty(Q_T)$, 
\item[2)] $\tilde u^\Delta \rightarrow \tilde u$ in $L^1(Q_T)$
and a.e. in $Q_T$,
\item[3)] $\psi^\Delta \rightarrow \psi \chi_{Q_T \backslash \widetilde \Omega}$ in $L^1(Q_T)$ and a.e. in $Q_T$,
\item[4)] $u^\Delta \rightarrow u$ in $L^1(\widetilde\Omega)$,
\item[5)] $\tilde u = u + \psi \chi_{Q_T \backslash \widetilde \Omega}$ and $u = \tilde  u \chi_{\widetilde \Omega}$.
\end{enumerate}

\end{lemma}
\begin{proof}
The first statement
follows from Lemma \ref{linftybound} and
Corollary \ref{cor1p}. 
The second statement follows from Corollary \ref{cor}. 
To prove the third statement we write
$$
 \|\psi^\Delta- \psi\chi_{Q_T\backslash \widetilde \Omega}\|_{L^1(Q_T)}\le \| \psi\|_{L^\infty(Q_T)} \|\chi_{Q_T\backslash \Omega^\Delta} - \chi_{Q_T\backslash \widetilde \Omega}\|_{L^1(Q_T)} \rightarrow 0
$$
as $N \to \infty$, thanks to Lemma \ref{basic_fact}. It follows that
$$
u^\Delta=\tilde u^\Delta-\psi^\Delta \rightarrow u:=\tilde u - \psi\chi_{Q_T\backslash \widetilde \Omega} \quad \mbox{in}\ L^1(\widetilde\Omega).
$$
Since $\Omega^\Delta \to \widetilde{\Omega}$ by Lemma \ref{basic_fact}, 
we get that $\tilde u^\Delta \to \psi$ a.e. in $Q_T\backslash \widetilde\Omega$, so that $u$ is supported on $\widetilde\Omega$. 
\end{proof} 

Recalling Lemma \ref{linftybound} it follows that, 
up to a subsequence,
$\tilde u^\Delta\to \tilde u$ in $L^p(Q_T)$ and
$u^\Delta\to u$ in $L^p(\widetilde\Omega)$,
for all $1\le p<\infty$.

We now discuss the convergence of the time derivatives.  
\begin{lemma}
\label{45}
There exists $\widetilde\Lambda \in \mathcal{D}'(Q_T)$ such that, up to extraction of a subsequence, $\tilde u^\Delta_t \rightharpoonup \widetilde\Lambda$ in $\mathcal{D}'(Q_T)$. In fact, $\widetilde\Lambda$ agrees as a distribution over $Q_T$ with the time derivative (in distributional sense) of the function $\tilde u$ defined in Lemma \ref{u_conv}. Moreover, given any cylinder $C:=(t_a,t_b)\times K \subset \subset \widetilde\Omega$, there holds that $\widetilde\Lambda_{|C} \in L^{p'}(0,T;W^{-1,p'}(K))$ and $\tilde u^\Delta_t|_C \rightharpoonup \widetilde\Lambda_{|C}$ in $L^{p'}(0,T;W^{-1,p'}(K))$. 
\end{lemma}
\begin{proof}
Let us denote by $\langle \cdot ,\cdot \rangle$ the pairing between $\mathcal{D}'(Q_T)$ and $\mathcal{D}(Q_T)$.
Given $\phi \in \mathcal{D}(Q_T)$, we compute 
$$
\langle \tilde u_t^\Delta,\phi \rangle = -\langle \tilde u^\Delta,\phi_t \rangle = -\int_0^T \int_{Q_0}\tilde u^\Delta \phi_t \, dxdt.
$$
We may now use Corollary \ref{cor} to pass to the limit, so that
$$
 -\int_0^T \int_{Q_0}\tilde u \phi_t \, dxdt = \lim_{N\to \infty} \langle \tilde u_t^\Delta,\phi \rangle = \langle \widetilde\Lambda,\phi \rangle
$$
up to a subsequence. This shows the first and second statements.

Our last statement is a consequence of Lemma \ref{duality_estimate}, which provides uniform bounds on the time derivative over cylinders contained in $\widetilde\Omega$ as in Lemma \ref{comp_time}.
\end{proof}
\begin{corollary}
Let $\Sigma_{t_1,t_2}:=[t_1,t_2] \times K$ such that $\Sigma_{t_1,t_2} \cap \overline{\partial \widetilde\Omega} = \emptyset$. Then $\tilde u \in C(t_1,t_2;L^2(K))$.
\end{corollary}
\begin{proof}
Let $\phi \in \mathcal{D}(\Sigma_{t_1,t_2})$. By previous considerations, we know that $(\phi \tilde u)_t \in L^{p'}(t_1,t_2;W^{-1,p'}(K))$ and also $(\phi\, \tilde u)(t) \in W_0^{1,p}(K)$ for a.e. $t_1< t< t_2$. Using Lemma \ref{interpol} we deduce that 
$\phi\, \tilde u \in C(t_1,t_2;L^2(K))$.
Being $\phi$ and $K$ arbitrary, the thesis follows.
\end{proof}
\begin{corollary}
\label{important}There holds that 
$\widetilde \Lambda_{|\widetilde\Omega}
 \in\mathcal{V}^*$.
\end{corollary}
\begin{proof}
Let $\phi \in C_c^1(\widetilde\Omega)$. Thanks to Lemma \ref{duality_estimate} we have that 
$$
| \langle \widetilde \Lambda,\phi \rangle_{\mathcal{V}^*-\mathcal{V}} |\le ||\phi||_{\mathcal{V}} \left(c \Vert \tilde u \Vert_{L^p(0,T,W^{1,p}(Q_0))}^{p-1} 
+ \Vert b\Vert_{L^p(0,T,L^{p'}(Q_0))}\right).
$$
Our claim follows by a  duality argument.
\end{proof}
%

\subsection{Recovery of the limit equation}
Our next aim is identifying the limit equation.  
Let us define 
$$
A^\Delta(t,x) := \sum_{k=0}^{N-1} \chi_{[t_k,t_{k+1})}(t) A(t_k,x,u^k,\nabla u^k)\chi_{\Omega(t_k)}(x).
$$

\begin{lemma}
\label{weak_flux}
There exists a function $\bar A \in L^{p'}(Q_T)^d$ such that 
$A^\Delta \rightharpoonup \bar A$ in $L^{p'}(Q_T)^d$ as $N \to \infty$,
up to a subsequence. 
Moreover $\bar A$ is supported in $\widetilde\Omega$.
\end{lemma}
\begin{proof}
This follows directly from \eqref{L1} and Lemma \ref{nabla_bound}.
\end{proof}
To identify $\bar A$ we will require a number of auxiliary results. 
\begin{lemma}
\label{auxtau}
Let $\phi$ be smooth and such that $\mbox{supp}\ \phi \subset {\Omega}^\Delta\cap \widetilde{\Omega}$. Given $\tau>0$ we define 
$$
\rho^\tau:= \frac{1}{\tau} \int_{t-\tau}^t ((\phi(t)-\phi(s)) u(s)\, ds
$$ 
(we set $\rho^\tau:= 0$ when the previous formula does not make sense), being $u$ the function defined in Lemma \ref{u_conv}. Then $\rho^\tau \in \mathcal{V}$ {for any $\tau >0$ and} $\rho^\tau\to 0$ in $\mathcal{V}$ as $\tau \to 0$.
\end{lemma}
\begin{proof}
Since $\mbox{supp}\, \rho^\tau \subset {\Omega}^\Delta\cap \widetilde{\Omega}$ for small $\tau$, we can approximate $\rho^\tau$ in the norm of $\mathcal{V}$ by functions in $C_c^1(\widetilde\Omega)$ convolving with a mollifying sequence,
so that $\rho^\tau\in\mathcal{V}$. 

Let now $K\subset \R^d$ be an open set such that $K \subset \Omega^\Delta(t)$ a.e. $t\in (t_a,t_b)$ for some values $0\le t_a<t_b\le T$. Thanks to 
\cite[Ch. 2, Th. 9]{VectorMeasures}, 
we get that $\rho^\tau \to 0$ in $L^p(t_a,t_b;W^{1,p}(K))$ as $\tau \to 0$. Covering $\mbox{supp}\, \phi$ with a finite collection of cylinders of the form $(t_a,t_b)\times K$ yields the desired result.
\end{proof}
\begin{lemma}\label{ext1}
Let $\phi$ be smooth and such that $\mbox{supp}\ \phi 
 \subset {\Omega}^\Delta\cap \widetilde{\Omega}$. 
Then
\begin{equation}\label{limitE1}
\limsup_{N\to\infty} \int_0^T\int_{\Omega^\Delta(t)} A^\Delta \cdot \nabla u^\Delta \phi\, dx dt \leq
\int_0^T\int_{\Omega(t)} \bar A \cdot \nabla u\, \phi\, dx dt.
\end{equation}
\end{lemma}
\begin{proof} Let $\tau >0$ and define
$$
u^\tau(t) = \frac{1}{\tau} \int_{t-\tau}^t u(s)\, ds.
$$
By multiplying the equation for $u^\Delta$ by $(u^\Delta-u^\tau) \phi$ and integrating by parts we get
$$
\int_0^T\int_{\Omega^\Delta(t)} (u^\Delta-u^\tau)  u^\Delta_t \phi \, dx dt = \int_0^T\int_{\Omega^\Delta(t)} \left(u^\Delta-u^\tau \right) \div A^\Delta \phi \, dx dt
$$
$$
=-\int_0^T\int_{\Omega^\Delta(t)} A^\Delta \cdot \nabla u^\Delta\phi\, dx dt
- \int_0^T\int_{\Omega^\Delta(t)} A^\Delta \cdot \nabla \phi\,  u^\Delta \, dx dt
$$
$$
+\int_0^T\int_{\Omega^\Delta(t)} A^\Delta \cdot \nabla u^\tau \phi\, dx dt
+ \int_0^T\int_{\Omega^\Delta(t)} A^\Delta \cdot \nabla \phi\, u^\tau \, dx dt:= I+II+III+IV.
$$
Let us elaborate on the left hand side of the previous equality. We compute
$$
\int_0^T\int_{\Omega^\Delta(t)} u^\Delta  u^\Delta_t \phi \, dx dt = \int_0^T\int_{\Omega^\Delta(t)} \phi \frac{\partial}{\partial t} \left[\frac{(u^\Delta(t))^2}{2}\right] \, dx dt
$$
$$
=- \int_0^T\int_{\Omega^\Delta(t)} \frac{(u^\Delta(t))^2}{2} \phi_t \, dx dt \to
- \int_0^T\int_{\Omega(t)} \frac{u^2}{2} \phi_t\, dx dt
$$
as $N \to \infty$, thanks to Lemma \ref{u_conv}. 
Next, we have that 
\begin{eqnarray*}
- \int_0^T\int_{\Omega^\Delta(t)}u^\tau  u^\Delta_t \phi\, dx dt &=& - \int_0^T\int_{\Omega^\Delta(t)} u^\Delta_t  \frac{\phi}{\tau} \int_{t-\tau}^t u(s)\, ds \, dx dt
\\
&=& - \int_0^T\int_{\Omega^\Delta(t)}u^\Delta_t \left\{(\phi \, u)^\tau +\frac{1}{\tau} \int_{t-\tau}^t ((\phi(t)-\phi(s)) u(s)\, ds \right\} \, dx dt
\\
&=& \int_0^T\int_{\Omega^\Delta(t)}(\phi u)^\tau_t u^\Delta \, dx dt - \int_0^T\int_{\Omega^\Delta(t)}\rho^\tau u^\Delta_t \, dx dt
\\
&=& \int_0^T\int_{\Omega^\Delta(t)}\frac{\phi(t)u(t)-\phi(t-\tau)u(t-\tau)}{\tau} u^\Delta \, dx dt - \int_0^T\int_{\Omega^\Delta(t)}\rho^\tau u^\Delta_t \, dx dt
\\
&=:& A+B.
\end{eqnarray*}
Thanks to our assumptions on $\phi$ we have that
$$
B=- \int_0^T \int_{\Omega^\Delta(t)}\rho^\tau \sum_{k=0}^{N-1} \chi_{[t_k,t_{k+1})} u^k_t\chi_{\Omega(t_k)} \, dxdt
$$
for $\tau$ small enough. We then pass to the limit in $B$ by Lebesgue's dominated convergence theorem. Indeed, if $\tau$ is small enough 
Lemma \ref{45} enables to get a.e. convergence of the integrand, domination follows as the duality product is uniformly bounded. To deal with the limit of $A$ as $N\to \infty$ we may use Lemma \ref{u_conv}(4) together with the fact that the incremental ratio is essentially bounded (after Lemma \ref{linftybound}). Gathering all the previous and letting $N\to \infty$, we find that 
$$
- \int_0^T\int_{\Omega^\Delta(t)}u^\tau  u^\Delta_t \phi\, dx dt \rightarrow \int_0^T\int_{\Omega(t)}\frac{\phi(t)u(t)-\phi(t-\tau)u(t-\tau)}{\tau} u \, dx dt 
- \int_{Q_T}\rho^\tau \widetilde\Lambda \, dx dt,
$$
 which is bounded from below by
$$
\int_0^T\int_{\Omega(t)}\frac{\phi(t)-\phi(t-\tau)}{\tau} \frac{u^2(t)}{2} \, dx dt - \int_{Q_T}\rho^\tau 
\widetilde\Lambda \, dx dt.
$$
Letting $\tau\to 0+$ and using Lemma \ref{auxtau}, we obtain
$$
 \int_0^T\int_{\Omega(t)}\phi_t \frac{u^2(t)}{2} \, dx dt,
$$
so that $\liminf_{\tau\to 0+} \liminf_{N \to \infty} (I+II+II+IV) \ge 0.$

We are now ready to compute the limit of $I+II+III+IV$ when $N \to \infty$. 
First, we find out that
$$
II \rightarrow -\int_0^T \int_{\Omega(t)} \bar A \nabla \phi\, u \, dx dt
$$
using Lemmas \ref{u_conv}(4) and \ref{weak_flux}. 
We also have
$$
III \rightarrow \int_0^T \int_{\Omega(t)} \bar A \nabla u^\tau \phi \, dx dt
$$
as $N \to \infty$ (clearly $\nabla u^\tau \in L^p(Q_T)^d$). Note that $\nabla u^\tau = (\nabla u)^\tau \to \nabla u$ in $L_{loc}^p(\widetilde\Omega)^d$, as in the proof of Lemma \ref{auxtau}. Taking limit $\tau \to 0$, the above integral converges to
$$
\int_0^T \int_{\Omega(t)} \bar A \nabla u\, \phi \, dx dt.
$$
Finally, arguing as before we get that
$$
IV \rightarrow \int_0^T \int_{\Omega(t)} \bar A \nabla \phi \, u^\tau \, dx dt
$$
as $N \to \infty$, 
which converges to 
$$
\int_0^T \int_{\Omega(t)} \bar A \nabla \phi \, u \, dx dt
$$
after taking the limit $\tau \to 0$. Hence 
$$
\limsup_{N \to \infty} \int_0^T \int_{\Omega(t)} A^\Delta \nabla u^\Delta \phi \, dx dt \le 
 \int_0^T \int_{\Omega(t)} \bar A \phi \nabla u \, dx dt
$$
and the result follows.
\end{proof}
\begin{lemma}
\label{MB}
There holds $\bar A(t,x) =A(t,x,u,\nabla u)$ a.e. in $\widetilde\Omega$.
\end{lemma}
\begin{proof}
We use Minty--Browder's technique.
 Let $0\le \phi \in C_0^1(Q_T)$ with 
${\rm supp}\,\phi\subset {\Omega}^\Delta\cap\widetilde{\Omega}$, and 
let $g \in C^1(\overline{Q_T})$.
Thanks to the monotonicity assumption \eqref{L3},
we have
$$
\int_0^T \sum_{k=1}^{N-1} \int_{\Omega(t_k)} (A(t_k,x,u^\Delta,\nabla u^\Delta)-A(t_k,x,u^\Delta,\nabla g)) (\nabla u^\Delta(t)-\nabla g)\phi \, dxdt \geq 0.
$$
From Lemma \ref{ext1} we get
$$
\limsup_{N\to\infty} \int_0^T \sum_{k=1}^{N-1} \int_{\Omega(t_k)} A(t_k,x,u^\Delta,\nabla u^\Delta) \nabla u^\Delta \phi\, dxdt \le 
\int_0^T \int_{\Omega(t)} \bar A \nabla u  \phi\, dxdt.
$$

We now show that
\begin{equation}
\label{expand}
\int_0^T \sum_{k=1}^{N-1} \int_{\Omega(t_k)} A(t_k,x,u^\Delta,\nabla g) 
\nabla u^\Delta\phi \, dxdt \rightarrow \int_0^T \int_{\Omega(t)} A(t,x,u,\nabla g) \nabla u \, \phi\, dxdt,
\end{equation}
as $N\to\infty$. Indeed, recalling \eqref{L4} we have
$$
\left|A(t,x,u,\nabla g) - \!\sum_{k=1}^{N-1}\! A(t_k,x,u^\Delta,\nabla g) 
\chi_{[t_k,t_{k+1})}  \right| \le \!
\sum_{k=1}^{N-1} \chi_{[t_k,t_{k+1})}\! \left( w(|t-t_k|)\!+\! C|u(t,x)\!-\!u^\Delta(t,x)| 
\right)\! |\nabla g|^{p-1}.
$$
Note that the right-hand side above converges to zero a.e. in $\widetilde \Omega$ and also in 
$L^p(\widetilde \Omega)$ for all $p<\infty$  
as $N \to \infty$. On the other hand, $\nabla u^\Delta \rightharpoonup \nabla u$ weakly in $L_{loc}^p(\widetilde\Omega)^d$ thanks to Lemma \ref{u_conv},  
which yields \eqref{expand}. In a similar way we show that
$$
\int_0^T \sum_{k=1}^{N-1} \int_{\Omega(t_k)} A(t_k,x,u^\Delta,\nabla g) 
\nabla g \, \phi \, dxdt \rightarrow \int_0^T \int_{\Omega(t)} A(t,x,u,\nabla g) \nabla g \, \phi\, dxdt.
$$
Finally we obtain that
$$
\int_0^T \sum_{k=1}^{N-1} \int_{\Omega(t_k)} A(t_k,x,u^\Delta,\nabla u^\Delta) 
\nabla g \, \phi \, dxdt \rightarrow \int_0^T \int_{\Omega(t)} \bar A \nabla g \, \phi\, dxdt
$$
thanks to Lemma \ref{weak_flux}. Summing up, we obtain
$$
\int_0^T \int_{\Omega(t)} (\bar A -A(t,x,u,\nabla g)) (\nabla u (t)-\nabla g)\phi\, dxdt \geq 0.
$$
This implies that $\bar A=A(t,x,u,\nabla u)$ for a.e. $(t,x)\in \mbox{supp} \, \phi$, by means of Minty--Browder's method (see for instance \cite[Ch. 9.1]{Evans}). 
\end{proof}
\subsection{Recovery of boundary and initial conditions}

\begin{proposition}
The function $u$ defined in Lemma \ref{u_conv} is a weak solution of problem \ref{model1} in the sense of Definition \ref{solucion_distribucional}. Furthermore, $u(t)\rightarrow u_0$ a.e. as $t\to 0$.
\end{proposition}
\begin{proof}
 Let $\phi\in C^\infty_0(Q_T)$ with 
$\mathrm{supp}\, \phi\subset {\Omega}^\Delta\cap \widetilde\Omega$. 
We fix a value of $k\in \{1,\ldots , N-1\}$ and test the approximating 
problem in  $[t_k,t)\times \Omega(t_k)$ with $t <t_{k+1}$. That is,
\begin{eqnarray*}
& & \int_{\Omega(t_k)} u^\Delta(t) \phi(t)\, dx + \int_{t_k}^{t} \int_{\Omega(t_k)} A^\Delta\cdot \nabla  \phi\, dx ds  \\
& = & \int_{\Omega(t_k)} u^\Delta(t_k) \phi(t_k)\, dx + \int_{t_k}^{t} \int_{\Omega(t_k)} u^\Delta(s)\phi_s \, dx ds
\end{eqnarray*}
for any $t \in [t_k,t_{k+1})$.
By adding these contributions from $0$ to $t\in (t_j,t_{j+1}], j \in \{1,\ldots , N-1\}$ we get
\begin{eqnarray}
& & \int_{\Omega^\Delta(t)} u^\Delta(t)\phi(t)\, dx 
+\int_{0}^{t} \int_{\Omega^\Delta(s)} A^\Delta\cdot \nabla \phi  \, dx ds 
\nonumber
\\
& = & \int_{\Omega(0)} u_0\phi(0)\, dx + \int_{0}^{t} \int_{\Omega^\Delta(s)} u^\Delta(s)\phi_s \, dx ds \label{esto} \\ %
& + &
\sum_{k=1}^j \left( \int_{\Omega(t_{k})} u^\Delta(t_k+)\phi(t_k)\, dx - \int_{\Omega(t_{k-1})} u^\Delta(t_k-)\phi(t_k)\, dx \right).
\nonumber
\end{eqnarray}
Since ${\rm supp}\, \phi\subset\Omega^\Delta$, we also have
\begin{eqnarray*}
& &   \int_{\Omega(t_{k})} u^\Delta(t_k+)\phi(t_k)\, dx - \int_{\Omega(t_{k-1})} u^\Delta(t_k-)\phi(t_k)\, dx  \\ & = &
\int_{\Omega(t_{k})\backslash \Omega(t_{k-1})} \psi(t_k)\phi(t_k)\, dx - \int_{\Omega(t_{k-1})\backslash \Omega(t_k)} u^\Delta(t_k-)\phi(t_k)\, dx =0.
\end{eqnarray*}
Thanks to Lemma \ref{u_conv}(4), $u^\Delta$ converges strongly to $u$ 
in $L^1(\mbox{supp}\, \phi)$. Hence we can pass to the limit 
in \eqref{esto} and obtain
\begin{equation}
\nonumber
\begin{array}{ll}
\displaystyle \int_{\Omega(t)} u(t)\phi(t)\, dx  +
\int_{0}^{t} \int_{\Omega(s)} A(t,x,u,\nabla u)\cdot \nabla \phi \, dx ds \\ \\
\displaystyle =  \int_{\Omega(0)} u_0\phi(0)\, dx + \int_{0}^{t} \int_{\Omega(s)} u(s)\phi_s \, dx ds
\end{array}
\end{equation}
for a.e. $0<t\le T$, which holds for any $\phi\in C^\infty_0(Q_T)$ 
with $\mathrm{supp}\, \phi\subset \widetilde\Omega$. 
This can be stated as
\begin{equation}
\nonumber
u_t = \div A(t,x,u,\nabla u)\quad \mbox{in}\ \mathcal{D}'(\widetilde\Omega).
\end{equation}
Furthermore, since 
$\tilde u \in L^p(0,T;W^{1,p}(Q_0))$ and $\tilde u = \psi$ a.e. $Q_T\backslash \widetilde \Omega$, 
we get that $u(t)-\psi(t) \in W_0^{1,p}(\Omega(t))$ for almost any $t \in (0,T)$. 
Hence we also recover the boundary conditions 
at $\partial_l \widetilde \Omega$ in the limit. 

Let us deal next with the initial condition. Note that for $t$ small enough we have
$$
\int_{\Omega(t)} u(t) \phi(t)\, dx = \int_{\Omega(0)} u_0 \phi(0)\, dx + C(\phi) t
$$
for some $C(\phi)>0$. Here we use that we assume condition \emph{\ref{punto4}} on the time slicing (and specifically on $t_0=0$) as specified at the beginning of  Section \ref{Sec:approximations}. 
 Hence
$$
\lim_{t\to 0} \int_{\Omega(t)} u(t) \phi(t)\, dx = \int_{\Omega(0)} u_0 \phi(0)\, dx.
$$
Now let $K\subset \subset \Omega(0)$ such that $\tilde u \in C(0,t_1,L^2(K))$ for some $t_1>0$ (which exists as $\widetilde\Omega$ is Lipschitz). Then $u(t)$ converges in $L^2(K)$ to some $\bar u_0$ as $t\to 0$. This limit $\bar u_0$ must agree with the distributional limit $u_0$ over $K$. Hence $u(t)\rightarrow u_0$ in $L_{loc}^2(\Omega(0))$ as $t\to 0$. In particular we get a.e. convergence to the initial condition. Note that this works in the same way for any relatively open subset of $\partial_{-1}\widetilde\Omega$.

Finally we justify that $u -\psi \in \mathcal{V}$. Once we have shown that the boundary conditions on $ \partial_l \widetilde\Omega$ are fulfilled, it is easy to construct a sequence $\eta_n$ belonging to $C_c^1(\widetilde\Omega)$ 
and satisfying $\|(u-\psi)-\eta_n\|_\mathcal{V} \to 0$ as $n \to \infty$. For instance, we may consider $G\in C^1(\R)$ such that $|G(t)|\le |t|$, $G(t)=0$ if $|t|\le 1$ and $G(t)=t$ if $|t|\ge 2$. We also consider $\rho_n$ to be a standard mollifying sequence. Then $\eta_n = G(n\rho_n \ast (u-\psi))/n$ has the desired properties.
\end{proof}

The argument above also shows that, given a cylinder $[t_1,t_2]\times K \subset \subset \widetilde\Omega$, the map $t\mapsto u_{|K}$ is $L^2$-continuous in $[t_1,t_2]$. As a consequence, if we fix $t>0$ then $u(s) \rightarrow u(t)$ as $s\to t$ a.e. in $\Omega(t)$. In this sense, we can claim that $t \mapsto u(t) \in C(0,T,L^2(\Omega(t)))$.

\section{Uniqueness of solutions}
\label{so_unique}

We start with a technical result which can been proved as in \cite[Proposition 2.6]{Paronetto}.

\begin{proposition}\label{propar}
Let Assumptions \ref{ass1}--\ref{ass4} and \ref{for_uniqueness} be satisfied. Then
the following integration by parts formula holds: 
\begin{equation}
\label{ibp}
\int_{t_1}^{t_2} \langle u_t,v\rangle_s+ \langle v_t,u\rangle_s \, ds = \int_{\Omega(t_2-)} u(t_2-) v(t_2-) \, dx - \int_{\Omega(t_1+)} u(t_1+) v(t_1+) \, dx,
\end{equation}
for any $0\le t_1 <t_2\le T$ and any $u,v\in \mathcal{V}$,
where $\langle \cdot,\cdot \rangle_t$ indicates the pairing between $W^{-1,p'}(\Omega(t))$ and $W_0^{1,p}(\Omega(t))$.
\end{proposition}

\medskip 

\noindent{\it Proof of Theorem \ref{teounique}.}
Let $\tilde u_1, \tilde u_2$ be two solutions of \eqref{model1}. 
Let $\eps>0$ and 
define
$$
g_\epsilon (x):= \left\{
\begin{array}{cc}
\mbox{sign}(x) \left(-\frac{5|x|^4}{16 \epsilon^4}-\frac{2|x|^3}{ \epsilon^3}-\frac{9|x|^2}{2 \epsilon^2}+\frac{4|x|}{\epsilon} \right) & |x|<	2\epsilon,
\\
\mbox{sign}(x) & |x|\ge 2\epsilon
\end{array}
\right.
\in C^2(\R),
$$
which is a regularization of the sign function that converges pointwise as $\epsilon \to 0$. 
Note also that we have $g_\epsilon(\tilde u_1 - \tilde u_2) \in L^p(0,T,W_0^{1,p}(Q_0))$. 
Besides, $\mbox{supp}\, g_\epsilon(\tilde u_1 - \tilde u_2)$ 
lies in the closure of $\widetilde \Omega$. Then, with a slight abuse of notation, $g_\epsilon(\tilde u_1 - \tilde u_2)=g_\epsilon( u_1 -  u_2)$.

We pick $\{\phi_n\}_n \in \mathcal{D}(Q_T)$ such that $\phi_n \rightarrow g_\epsilon( u_1- u_2)$ strongly in $L^p(0,T,W_0^{1,p}(Q_0))$ and $\mbox{supp}\ \phi_n \subset \widetilde\Omega$. Note that 
the pairing 
$$
\langle ( u_1- u_2)_t,\phi_n \rangle_{\mathcal{V}^*-\mathcal{V}}
$$
makes sense and is bounded independently of $n$. Then we substitute $\phi_n$ in \eqref{weakId1}. On one hand, when $n \to \infty$ we get
$$
 \int_{\widetilde\Omega} \phi_n ( u_1- u_2)_t \, dxdt \rightarrow  \int_{\widetilde\Omega} g_\epsilon( u_1- u_2) ( u_1- u_2)_t \, dxdt.
$$
On the other hand, integrating by parts and using \eqref{L1},
$$
 \int_{\widetilde\Omega} \phi_n ( u_1-u_2)_t \, dxdt  = -  \int_{\widetilde\Omega} \nabla \phi_n \left( A(t,x, u_1,\nabla  u_1) - A(t,x, u_2,\nabla  u_2)\right)\, dxdt
$$
$$
\rightarrow -  \int_{\widetilde\Omega} \nabla g_\epsilon( u_1 -  u_2) \left( A(t,x, u_1,\nabla u_1) - A(t,x, u_2,\nabla  u_2)\right)\, dxdt \quad \mbox{as}\, n \to \infty.
$$
Thus, we have shown that
$$
\langle ( u_1- u_2)_t,g_\epsilon( u_1 -  u_2) \rangle_{\mathcal{V}^*-\mathcal{V}} =-  \int_{\widetilde\Omega} g_\epsilon^{'}( u_1- u_2) \nabla( u_1- u_2) [A(t,x, u_1,\nabla u_1) - A(t,x, u_2,\nabla  u_2)] \, dxdt.
$$
Using the fact that
$$
[g_\epsilon( u_1- u_2)]_t = g_\epsilon'( u_1- u_2) \cdot ( u_1- u_2)_t\quad \mbox{in}\ \mathcal{D}'(\widetilde\Omega)
$$
and denoting
$$
p_\epsilon(x):= \left\{
\begin{array}{cc}
x g_\epsilon'(x) & x\in (-2\epsilon,2 \epsilon)
\\ \\
0 & |x|\ge 2\epsilon
\end{array}
\right.
\in C^1(\R)
$$
 we may argue as before to obtain that
$$
\langle [g_\epsilon( u_1- u_2)]_t, u_1- u_2\rangle_{\mathcal{V}^*-\mathcal{V}} = 
-  \int_{\widetilde\Omega}\nabla [p_\epsilon (u_1-u_2)]  [A(t,x, u_1,\nabla u_1) - A(t,x, u_2,\nabla  u_2)]\, dxdt.
$$
In such a way,
$$
 \langle ( u_1- u_2)_t,g_\epsilon( u_1 -  u_2) \rangle_{\mathcal{V}^*-\mathcal{V}}+ \langle [g_\epsilon( u_1- u_2)]_t, u_1- u_2\rangle_{\mathcal{V}^*-\mathcal{V}}
$$
$$
= - \int_{\widetilde\Omega}  \nabla( u_1- u_2) [A(t,x, u_1,\nabla  u_1) - A(t,x, u_1,\nabla  u_2)]\left\{ g_\epsilon'(u_1- u_2)+p_\epsilon'(u_1-u_2)\right\}\, dxdt
$$
$$
 - \int_{\widetilde\Omega}  \nabla( u_1- u_2) [A(t,x, u_1,\nabla  u_2) - A(t,x, u_2,\nabla  u_2)]\left\{ 2g_\epsilon'(u_1- u_2)+(u_1-u_2) g_\epsilon^{''}(u_1-u_2)\right\}\, dxdt
$$
The first term above is less or equal than zero due to \eqref{L3} and the fact that $g_\epsilon'+p_\epsilon'\ge 0$, hence we can neglect it. As regards the second term, we notice that there is some $C>0$ such that
$$
|g_\epsilon'(x)|\le C/\epsilon,\quad |x g_\epsilon^{''}(x)|\le C/\epsilon\quad \forall x\in -(2\epsilon,2\epsilon).
$$
Then we use \eqref{L4} to write 
$$
II\le \frac{2C}{\epsilon} \int_{\widetilde\Omega} \chi_{\{|u_1- u_2|<2 \epsilon\}} |\nabla( u_1- u_2)|  |\nabla  u_2|^{p-1} | u_1- u_2| \, dxdt
$$
$$
\le 4C \int_0^T \int_{| u_1 -  u_2|\le 2\epsilon} |\nabla( u_1- u_2)|  |\nabla  u_2|^{p-1} \, dxdt:=\theta(\epsilon), 
$$
which is uniformly bounded with respect to $\epsilon$. In fact this term vanishes in the limit $\epsilon \to 0$ given that $\nabla( u_1- u_2)=0$ almost everywhere on the set of points such that $u_1-u_2=0$. 
Then, thanks to \eqref{ibp} we obtain that
$$
\int_{\Omega(T-)} g_\epsilon( u_1 - u_2)(T-) (u_1 - u_2)(T-)\, dx - \int_{\Omega(0)} g_\epsilon( u_1 -  u_2)(0) ( u_1 -  u_2)(0)\, dx \le \theta(\epsilon) 
$$
and thus taking the limit $\eps \to 0$ we find
$$
\int_{\Omega(T-)} |u_1-u_2|(T-) \, dx \le \int_{\Omega(0)} |u_1-u_2|(0) \, dx
$$
for any $T>0$. This implies our uniqueness result.
\qed
\begin{remark}\rm
This proof can be considerably simplified if the operator $A$ does not depend  explicitly on $u$, as we can choose $g_\epsilon(x)=x$ in the previous computations and all the proof boils down to the monotonicity property \eqref{L3}. 
\end{remark}

\begin{remark}\rm
Let us note that the same uniqueness proof can be extended to the case in which there exists a finite number of times $t_0:=0<t_1<\ldots <t_{N-1} <t_N:=T$ such that $ ((t_i,t_{i+1})\times Q_0 )\cap \widetilde \Omega$ verifies Assumption \ref{for_uniqueness} for each $i=0,\ldots,N-1$. Namely, the former proof would show that any two solutions $u_1,u_2$ with the same initial datum agree on $ ((0,t_1)\times Q_0) \cap \widetilde \Omega$. Taking traces at $t_1-$ we find that $u_1=u_2$ a.e. on $\Omega(t_1-)$. Thus $u_1=u_2$ a.e. on $\Omega(t_1+)$ and we can repeat the former uniqueness proof to obtain that $u_1$ agrees with $u_2$ on $((t_1,t_2)\times Q_0) \cap \widetilde \Omega$ and hence on $(0,t_2)\times Q_0 \cap \widetilde \Omega$. We can continue in this way until we reach uniqueness in the whole of $\widetilde \Omega$. 
\label{finite_sing}
\end{remark}

\begin{remark}\rm
We observe that  
Assumption \ref{for_uniqueness} could be replaced by the more general requirement that the domain $\widetilde\Omega$ satisfies \eqref{ibp}. In fact, it suffices to have \eqref{ibp} with a "$\ge$" instead of "$=$", and only for functions $u,v\in \mathcal{V}$ such that $u\, v \ge 0$.
\end{remark}

\end{document}